\documentclass[12pt]{article}
\usepackage{amsfonts}
\usepackage{amsthm}
\usepackage{epsfig,ifpdf,xspace}
\ifpdf
\fi
\newtheorem{theorem}{Theorem}
\newtheorem{lemma}[theorem]{Lemma}
\newtheorem{corollary}[theorem]{Corollary}
\newtheorem{conjecture}{Conjecture}
\def\mgraph{abstract map\xspace}
\def\mgraphs{abstract maps\xspace}
\def\NN{{\mathbb N}}

\def\WW{{\cal W}}
\begin{document}
\title{Circular edge-colorings of cubic graphs with girth six}
\author{Daniel Kr{\'a}l'\thanks{Institute for Theoretical Computer Science (ITI),
        Faculty of Mathematics and Physics, Charles University,
	Malostransk\'e n\'am\v est\'\i{} 25, 118 00 Prague 1, Czech
    Republic. This research was partially supported by the
    grant GACR 201/09/0197.
	E-mail: \texttt{kral@kam.mff.cuni.cz}.
	Institute for Theoretical Computer Science is supported as project 1M0545 by Czech Ministry of Education.}\and
        Edita M{\'a}{\v c}ajov{\'a}\thanks{Department of Computer Science,
	Faculty of Mathematics, Physics and Informatics,
        Comenius University, Mlynsk{\'a} dolina, 842 48 Bratislava, Slovakia.
        This author's work was partially supported by the APVV project
        0111-07.
	E-mail: \texttt{macajova@dcs.fmph.uniba.sk}.}\and
        J\'an Maz{\'a}k\thanks{Department of Computer Science,
	Faculty of Mathematics, Physics and Informatics,
        Comenius University, Mlynsk{\'a} dolina, 842 48 Bratislava, Slovakia.
        This author's work was partially supported by the grant
        UK/384/2009 and by the APVV project 0111-07.
	E-mail: \texttt{mazak@dcs.fmph.uniba.sk}.}\and
	Jean-S\'ebastien Sereni\thanks{CNRS (LIAFA,
	Universit\'e Denis Diderot), Paris, France, and
    Department of Applied Mathematics (KAM) - Faculty of Mathematics and
    Physics, Charles University, Prague, Czech Republic. This author's
    work was partially supported by the \'Egide \textsc{eco-net} project
    16305SB.
	E-mail: \texttt{sereni@kam.mff.cuni.cz}.}}
\date{}
\maketitle
\begin{abstract}
We show that
the circular chromatic index
of a (sub)cubic graph with odd-girth at least $7$ is at most $7/2$.
\end{abstract}

\section{Introduction}

A classical theorem of Vizing~\cite{bib-vizing64} asserts that
the \emph{chromatic index} of every cubic bridgeless graph, i.e.,
the smallest number of colors needed to properly edge-color
such a graph, is $3$ or $4$. Cubic cyclically $4$-edge-connected
graphs with chromatic
index $4$ are known as \emph{snarks} and it is known that smallest
counter-examples, if any, to several deep open conjectures in graph theory, such
as the Cycle Double Cover Conjecture,
must be snarks.

Our research is motivated by edge-colorings of cubic bridgeless
graphs with no short cycles.
The Girth Conjecture of Jaeger and Swart~\cite{bib-jaeger80+} asserted that
there are no snarks with large girth. This conjecture was
refuted by Kochol~\cite{bib-kochol96} who constructed snarks
of arbitrary large girth. Hence, it is natural to ask whether
it can be said that snarks of large girth are close
to being $3$-edge-colorable in some sense.

One of the relaxations of ordinary colorings are circular
colorings, introduced by Vince~\cite{bib-vince88}. A \emph{$(p,q)$-coloring}
of a graph $G$ is a coloring of vertices with colors
from the set $\{1,\ldots,p\}$ such that
any two adjacent vertices receive colors $a$ and $b$
with $q\le |a-b|\le p-q$.
Circular colorings naturally appear
in different settings, which is witnessed by several equivalent definitions
of this notion as exposed in the surveys by Zhu~\cite{bib-zhu01,bib-zhu06}. 

The infimum of the ratios $p/q$ such that $G$
has a $(p,q)$-coloring is the \emph{circular chromatic number} of $G$.
It is known that the infimum is the minimum for all finite graphs and
the ceiling of the circular chromatic number of a graph is equal
to its chromatic number. Thus, the circular chromatic number
is a fractional relaxation of the chromatic number.
The \emph{circular chromatic index}
of a graph is the circular chromatic number of its line-graph.

Zhu~\cite{bib-zhu01} asked
whether there exist snarks with circular chromatic
index close or equal to $4$, and
as there are snarks with arbitrary large girth,
it is also interesting to know whether there exist such snarks of arbitrary
large girth.
Afshani et al.~\cite{bib-afshani05+} showed that the circular chromatic
index of every cubic bridgeless graph is at most $11/3$ and
Kaiser et al.~\cite{bib-kaiser04+} showed that for every $\varepsilon>0$,
there exists $g$ such that every cubic bridgeless graph with girth at least $g$
has circular chromatic index at most $3+\varepsilon$. This latter result
was generalized to graphs with bounded maximum degree~\cite{bib-kaiser07+}.
Moreover,
the circular chromatic indices of several well-known classes of snarks
have been
determined~\cite{bib-ghebleh06+,bib-ghebleh07,bib-ghebleh08,bib-mazak08}.

The Petersen graph is the only cubic bridgeless graph that is known
to have the circular chromatic index equal to $11/3$. In fact, it is
the only example of a cubic bridgeless graph with circular chromatic
index greater than $7/2$. This leads to the following (wild) conjecture.

\begin{conjecture}
\label{conj}
Every cubic bridgeless graph different from the Petersen graph
has circular chromatic index strictly less than $11/3$, maybe,
at most $7/2$.
\end{conjecture}

In their paper, Kaiser et al.~\cite{bib-kaiser04+},
formulated a problem to determine the smallest girth $g_0$ such that
every cubic bridgeless graph with girth at least $g_0$ has circular
chromatic index at most $7/2$, and they showed that $g_0\le 14$.
Note that $g_0\ge 6$ because of the Petersen graph. In this paper,
we prove that every cubic bridgeless graph having a $2$-factor composed
of cycles of length different from $3$ and $5$ has circular chromatic index at most
$7/2$.
This implies that $g_0=6$, i.e., the circular chromatic index of every
cubic bridgeless graph with girth at least $6$ is at most $7/2$.
Our result also applies to subcubic graphs with odd-girth at least $7$.

\section{Compatible trails}

The core of our argument is formed by decomposing
the graph obtained by contracting a $2$-factor of a cubic bridgeless graph
into trails. We first introduce notation related to such decompositions and
then prove their existence.

\subsection{Notation}

An \emph{\mgraph} is a graph with multiple edges, loops and half-edges
allowed with a fixed cyclic ordering of the ends around each vertex.
Formally, an \mgraph $(V,E,\varphi)$ is comprised of a vertex-set
$V$ and an edge-set
$E$. Each edge has two ends: one of them is incident with a vertex, and
the other may, but need not, be incident with a vertex. An edge that has
exactly one end incident with a vertex is a \emph{half-edge}. An edge
with both ends incident with the same vertex is a \emph{loop}.
The \emph{degree} $d_v$ of a vertex $v\in V$ is the number of ends
incident with $v$ (in particular, loops are counted twice).
Moreover, for each vertex,
there is a cyclic ordering of the ends incident with it. These orderings
are represented by a surjective mapping $\varphi:V\times\mathbb{N}\to E$
such that $\varphi(v,n)$ is an edge incident with $v$ and
$\varphi(v,n)=\varphi(v,n+d_v)$ for every $(v,n)\in
V\times\mathbb{N}$. An edge has two pre-images
in $\{v\}\times\{1,\ldots,d_v\}$ if and only if it is a loop incident
with $v$.
An \mgraph naturally yields an embedding of the corresponding
multi-graph on a surface.

\emph{Deleting a vertex $v\in V$} from
an \mgraph $G=(V,E,\varphi)$ yields an \mgraph $G'=(V',E',\varphi')$
with $V'=V\setminus\{v\}$,
$E'=\varphi(V'\times\NN)$ and $\varphi'$ being the restriction of $\varphi$
to $V'\times\NN$. In other words, the vertex $v$ and
half-edges and loops incident with $v$
are removed. The other edges incident with $v$ become half-edges.

A \emph{trail} in an \mgraph is a sequence of mutually distinct edges
$e_1,e_2,\ldots,e_k$
such that
\begin{itemize}
\item $e_{i}$ and $e_{i+1}$ have a common end-vertex $v_i$, for
$i\in\{1,\ldots,k-1\}$; and
\item $v_{i-1}\neq v_i$ unless $e_i$ is a loop, for $i\in\{2,\ldots,k-1\}$.
\end{itemize}
If $e_1$ is not a half-edge, then the trail \emph{starts
at a vertex}. Similarly, if $e_k$ is not a half-edge, the trail
\emph{ends at a vertex}. Furthermore, a trail that consists of a single
half-edge either starts or ends at a vertex.
For a given trail,
a linear ordering on its edges is naturally defined.
Let $T_1$ and $T_2$ be two edge-disjoint trails and $v$ a vertex.
If $T_1$ ends at $v$ and $T_2$ starts at $v$,
then we can \emph{link} the two trails by identifying
the end of $T_1$ with the beginning of $T_2$: thereby, we obtain
a new trail that first follows the edges of $T_1$ and then those
of $T_2$.
We also define the linking of a trail $T$ that starts
and ends at the same vertex with itself: in this case, we obtain the
same trail $T$ except that the ordering of the edges becomes cyclic.
A trail obtained in this way is \emph{closed}.
Trails that are not closed are \emph{open}. Note
that a trail that starts and ends at the same vertex can be either
open or closed.

Trails $W_1,\ldots,W_K$ form a \emph{compatible decomposition} of an \mgraph $G$
if all of them are open, every edge is contained in exactly one of the trails and for every vertex $v$
of odd degree of $G$, there is an index $i_v$
such that the following pairs of edges
are consecutive (regardless of their order) in some of the trails
(and thus are not the same edge):
\begin{itemize}
\item $\varphi(v,i_v)$ and $\varphi(v,i_v+5)$;
\item $\varphi(v,i_v+1)$ and $\varphi(v,i_v+3)$; and
\item $\varphi(v,i_v+2)$ and $\varphi(v,i_v+4)$.
\end{itemize}
Note that if an \mgraph $G$ has a compatible decomposition, then it has
no vertices of degree $1$ or $5$.

\subsection{Existence}

We now prove that every \mgraph with no vertices of degree $1$, $3$ or
$5$
has a compatible decomposition.
In the next two lemmas, which form the base of our inductive argument,
\mgraphs with a single vertex are analyzed.

\begin{lemma}
\label{lm-7}
Every \mgraph $G$ that has a single vertex and the degree of this vertex
is $7$ has a compatible decomposition.
\end{lemma}

\begin{proof}
Let $v$ be the vertex and let $e_i$ be the image of $\varphi(v,i)$
for each $i\in\{1,2,\ldots,7\}$.
Further, let $\WW_i$ be the set of trails obtained by the following
process: initially, each edge of $G$ forms a single trail. We then
link the two trails starting or ending with $e_i$ and $e_{i+5}$,
next we link the two trails starting or ending with $e_{i+1}$ and $e_{i+3}$, and
last the two trails starting or ending with $e_{i+2}$ and $e_{i+4}$.
This operation does not always
yield a compatible decomposition, since it may create closed trails. For
instance, $\WW_i$ contains a closed trail if $e_i=e_{i+5}$, or if
$e_{i+1}=e_{i+4}$ and $e_{i+2}=e_{i+3}$ (in this last case, after
having linked $e_{i+1}$ with $e_{i+3}$, thereby obtaining a trail
$T$, the last linking amounts to linking $T$ with itself).
On the other hand, if all the obtained trails are open, then
$\WW_i$ is a compatible decomposition, with $i$ being the
index $i_v$ of the definition.

Our goal is to show that at least one of the sets $\WW_i$
is composed only of open trails.
Since the degree of $v$ is odd, the vertex $v$
is incident with at least one half-edge.
By symmetry, we assume that this half-edge is $e_4$.

\begin{figure}
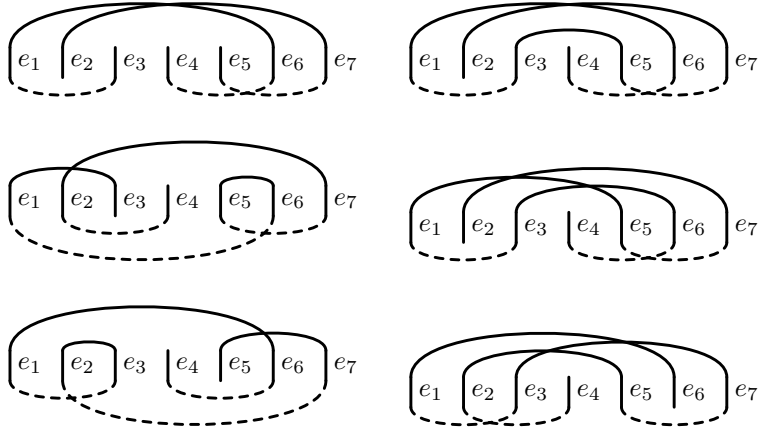

\begin{center}
\epsfbox{cubg6.11}
\hskip 5mm
\epsfbox{cubg6.12}
\vskip 5mm
\epsfbox{cubg6.13}
\hskip 5mm
\epsfbox{cubg6.14}
\vskip 5mm
\epsfbox{cubg6.15}
\hskip 5mm
\epsfbox{cubg6.16}
\end{center}
\caption{Compatible decompositions in the first (the top two pictures) and
         the last four (the remaining pictures) main cases
         in the proof of Lemma~\ref{lm-7}. The edges incident with the vertex
	 $v$ are drawn with solid lines and the way in which they are
	 joined to form the trails of the decomposition is indicated
	 by dashed lines.}
\label{fig-7-A}
\end{figure}

If $\WW_1$ contains a closed trail, then one of the following four cases
applies (recall that $e_4$ is a half-edge):
\begin{itemize}
\item $e_1=e_6$,
\item $e_3=e_5$,
\item $e_1=e_3$ and $e_5=e_6$, or
\item $e_1=e_5$ and $e_3=e_6$.
\end{itemize}
Similarly, if $\WW_2$ contains a closed trail, then one of the following
four cases applies:
\begin{itemize}
\item $e_2=e_7$,
\item $e_3=e_5$,
\item $e_2=e_3$ and $e_5=e_7$, or
\item $e_2=e_5$ and $e_3=e_7$.
\end{itemize}
Comparing the two sets of four possible cases, we conclude that if none
of $\WW_1$ and $\WW_2$ is compatible, then at least one
of the following six cases applies
(these cases are referred to as \emph{main} cases in Figures~\ref{fig-7-A}
and \ref{fig-7-B}):
\begin{itemize}
\item $e_1=e_6$ and $e_2=e_7$,
\item $e_3=e_5$,
\item $e_1=e_3$, $e_5=e_6$ and $e_2=e_7$,
\item $e_1=e_5$, $e_3=e_6$ and $e_2=e_7$,
\item $e_1=e_6$, $e_2=e_3$ and $e_5=e_7$, or
\item $e_1=e_6$, $e_2=e_5$ and $e_3=e_7$.
\end{itemize}
In the first case, the set $\WW_3$ contains only open trails.
In the last four cases, the sets $\WW_4$, $\WW_3$, $\WW_6$ and
$\WW_7$, respectively,
are composed of open trails (see Figure~\ref{fig-7-A} for an illustration).

\begin{figure}
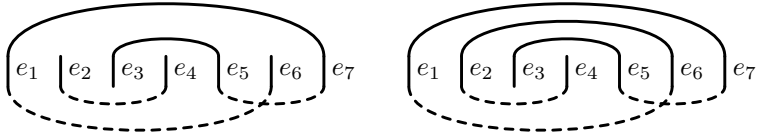

\begin{center}
\epsfbox{cubg6.21}
\hskip 5mm
\epsfbox{cubg6.22}
\end{center}
\caption{Compatible decompositions in the second main case
         in the proof of Lemma~\ref{lm-7}.}
\label{fig-7-B}
\end{figure}

We now focus on the second case. If $\WW_3$ contains a closed trail (and
$e_3=e_5$), it holds that $e_1=e_7$. However, in this case,
the set $\WW_4$ contains no closed trails (see Figure~\ref{fig-7-B}). This finishes
the proof of the lemma.
\end{proof}

\begin{lemma}
\label{lm-single}
Every \mgraph $G$ that has a single vertex and the degree of this vertex
is different from $1$, $3$ and $5$ has a compatible decomposition.
\end{lemma}

\begin{proof}
Let $v$ be the only vertex of the graph $G$. If the degree of $v$ is even,
then there is nothing to prove as every decomposition into open trails is
compatible. If the degree of $v$ is $7$, then the statement follows
from Lemma~\ref{lm-7}. Hence, we assume that the degree of $v$ is odd and
it is at least $9$. Let $e_i$ be the image of $\varphi(v,i)$
for $i\in\{1,2,\ldots,9\}$.

As the degree of $v$ is odd, $v$ is incident with at least one half-edge.
Without loss of generality, we can assume that $e_5$ is the half-edge.
Let $\WW_i$ be the set of trails defined as in the proof of Lemma~\ref{lm-7}.
Assume that both the sets $\WW_2$ and $\WW_3$
contain a closed trail; if any of them were composed of open trails only, then it would form a compatible
decomposition.

\begin{figure}
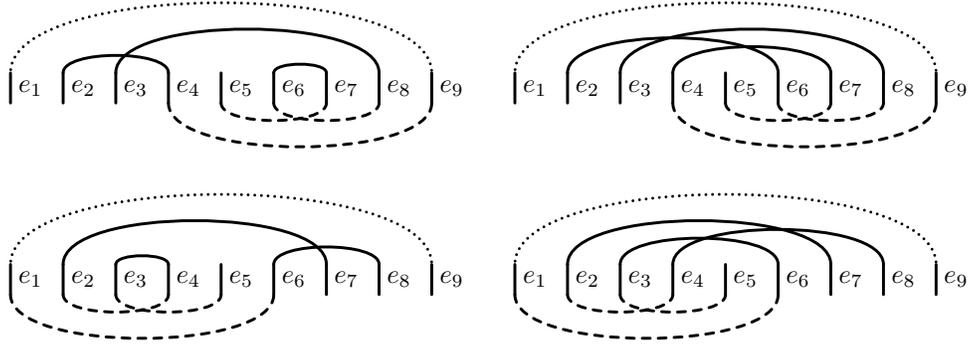

\begin{center}
\epsfbox{cubg6.31}
\hskip 5mm
\epsfbox{cubg6.32}
\vskip 5mm
\epsfbox{cubg6.33}
\hskip 5mm
\epsfbox{cubg6.34}
\end{center}
\caption{Compatible decompositions in the last four cases
         in the proof of Lemma~\ref{lm-single}. Additional
	 loops that can also be present are drawn with dotted lines.}
\label{fig-single-A}
\end{figure}

As in the proof of Lemma~\ref{lm-7}, we infer from the facts that
both $\WW_2$ and $\WW_3$ contain a closed trail that one of the following
six cases applies (replace $e_4$ with $e_5$ and $\WW_i$ with $\WW_{i+1}$
for $i\in\{1,2\}$ in the analysis done in the proof of Lemma~\ref{lm-7}):
\begin{itemize}
\item $e_2=e_7$ and $e_3=e_8$,
\item $e_4=e_6$,
\item $e_2=e_4$, $e_6=e_7$ and $e_3=e_8$,
\item $e_2=e_6$, $e_4=e_7$ and $e_3=e_8$,
\item $e_2=e_7$, $e_3=e_4$ and $e_6=e_8$, or
\item $e_2=e_7$, $e_3=e_6$ and $e_4=e_8$.
\end{itemize}
In the last four cases, the sets $\WW_4$, $\WW_4$, $\WW_1$ and $\WW_1$,
respectively, contain no closed trails (see Figure~\ref{fig-single-A} for an illustration).
Let us focus on the first two cases, now.

\begin{figure}
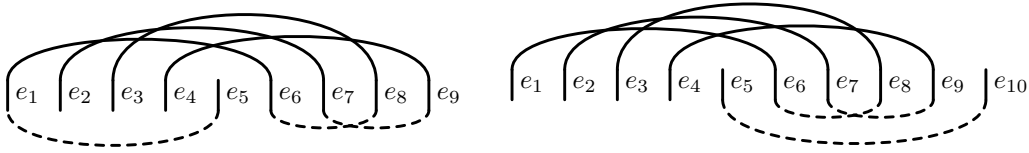

\begin{center}
\epsfbox{cubg6.41}
\hskip 5mm
\epsfbox{cubg6.42}
\end{center}
\caption{Compatible decompositions in the case where $e_2=e_7$ and $e_3=e_8$
         in the proof of Lemma~\ref{lm-single}. Two cases are distinguished
	 based on whether the degree of $v$ is equal to $9$ or not.}
\label{fig-single-B}
\end{figure}

Suppose that $e_2=e_7$ and $e_3=e_8$. If $\WW_1$ contains a closed trail,
then $e_1=e_6$. Further, if $\WW_4$ contains a closed trail,
then $e_4=e_9$. Thus, the set $\WW_5$
contains no closed trails (see Figure~\ref{fig-single-B}), and hence is
a compatible decomposition.

\begin{figure}
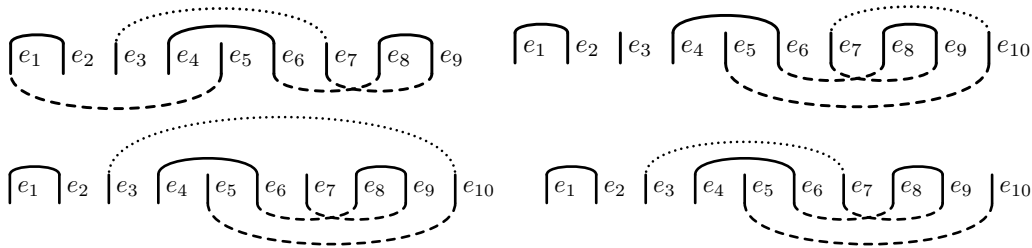

\begin{center}
\epsfbox{cubg6.51}
\hskip 5mm
\epsfbox{cubg6.52}\\
\epsfxsize=6.5cm
\epsfbox{cubg6.53}
\hskip 5mm
\epsfxsize=6.5cm
\epsfbox{cubg6.54}
\end{center}
\caption{Compatible decompositions in the case where $e_4=e_6$
         in the proof of Lemma~\ref{lm-single}. Two cases are distinguished
	 based on whether the degree of $v$ is equal to $9$ or not.
	 Additional loops that can also be present are drawn with dotted lines.}
\label{fig-single-C}
\end{figure}

Assume now that $e_4=e_6$. If $\WW_1$ contains a closed trail, then
$e_1=e_2$. Further, if $\WW_4$ contains a closed trail, then
$e_8=e_9$. Again, the set $\WW_5$
is then a compatible decomposition (see Figure~\ref{fig-single-C}).

To summarize, we have shown that at least one of the sets $\WW_i$,
$i\in\{1,2,3,4,5\}$, is composed of open trails only, and thus it forms a compatible
decomposition.
\end{proof}

Using Lemma~\ref{lm-single}, we show that every \mgraph with no
vertices of degree $1$, $3$ or $5$ has a compatible decomposition.

\begin{lemma}
\label{lm-mgraph}
Every \mgraph $G$ without vertices of degree $1$, $3$ or $5$
has a compatible decomposition.
\end{lemma}

\begin{proof}
The proof proceeds by induction on the number of vertices of $G$.
If $G$ has one vertex, then the statement follows from Lemma~\ref{lm-single}.
Otherwise, let $v$ be an arbitrary vertex of $G$ and let $G'$ be an \mgraph
obtained from $G$ by removing $v$. By the induction hypothesis, $G'$ has a compatible
decomposition $\WW'$ into trails.

Let $\WW$ be the set of trails obtained from $\WW'$ by adding the set
of loops and half-edges incident with $v$ in $G$.
Observe that $\WW$ is a set of trails of $G$ in which there is no
trail ``traversing'' $v$.
If the degree of $v$ is even, the set $\WW$ is a compatible
decomposition of $G$ as there is no restriction on how
the trails pass through the vertex $v$. Assume that the degree $d_v$
of $v$ is odd.

We now define an auxiliary \mgraph $H$. The \mgraph $H$ contains a single
vertex $w$ of degree $d_v$, and for $(i,j)\in\{1,2,\ldots,d_v\}^2$,
we have $\varphi(w,i)=\varphi(w,j)$
if and only if $\WW$ contains a trail starting with the edge $\varphi(v,i)$
and ending with $\varphi(v,j)$. In other words, trails of $\WW'$ that
both start and finish with an edge incident with $v$ correspond in $H$
to loops incident with $w$, the loops incident with $v$ are preserved and
the trails starting or finishing at $v$ (but not both) correspond
to half-edges. Trails containing no edge incident with $v$ have
no counterparts among the edges of $H$.

By Lemma~\ref{lm-single}, the \mgraph $H$ has a compatible decomposition $\WW_H$.
We can now obtain a compatible decomposition $\WW_G$ of $G$ as follows:
all the trails of $\WW$ neither starting nor ending at $v$ are added
to $\WW_G$.
Each trail $W$ of $\WW_H$ has a corresponding trail in $\WW_G$ that
is obtained by replacing every edge of $W$ with the corresponding
trail of $\WW$ and linking these trails.

It is straightforward to verify that $\WW_G$ is a set of open trails of $G$.
By induction, the trails pass through vertices of $G$ different
from $v$ in the way required by the definition of a compatible
decomposition. The trails also pass through $v$ in the required way
because $\WW_H$ is a compatible decomposition of $H$. Hence, $\WW_G$
is a compatible decomposition of $G$.
\end{proof}

\section{Graphs with odd-girth at least $7$}

We are now ready to prove our main theorem.

\begin{theorem}
\label{thm-main}
The circular chromatic index of every cubic graph with a $2$-factor
composed of cycles of lengths different from $3$ and $5$ is at most $7/2$.
\end{theorem}

\begin{proof}
Let $G$ be a cubic graph and $F$ a $2$-factor of $G$ composed of cycles
of lengths different from $3$ and $5$, and let $M$ be the perfect matching complementary to
$F$. The multi-graph obtained by contracting $F$ can be viewed
as an \mgraph $H$: the vertices of $H$ correspond to the cycles of
the $2$-factor $F$, and the order in which the edges of $M$ are incident
with cycles of $F$ naturally defines the function $\varphi$. Note that
$H$ has no half-edges and its loops correspond to chords of cycles of $F$.

Since no cycle of $F$ has length $3$ or $5$, no vertex of $H$ has degree
$1,3$ or $5$. Thus, by Lemma~\ref{lm-mgraph}, the \mgraph $H$ has a compatible
decomposition $\WW$. Color the edges of every trail of $\WW$ with $0$ and $1$
in an alternating way. Since the edges of $H$ correspond one-to-one
to the edges of $M$, we have obtained a coloring of the edges of $M$
with $0$ and $1$.

We now construct a $(7,2)$-edge-coloring of the edges of $G$.
Let $C=v_0v_1\cdots v_{\ell-1}$ be a cycle of $F$ and
$c_i$ the color of the edge of $M$
incident with the vertex $v_i$, for $i\in\{0,1,\ldots,\ell-1\}$.
If the length $\ell$ of $C$
is even, we color the edges of $C$ with $3$ and $5$ in an alternating way.
Let us consider the case where $\ell$ is odd. Since $\WW$ is a compatible
decomposition, there exists an index $k$ such that $c_k\neq c_{k+5}$,
$c_{k+1}\not=c_{k+3}$ and $c_{k+2}\not=c_{k+4}$ (indices are taken
modulo $\ell$) as the colors of the edges of trails of $\WW$ alternate.

We now show that there exists an index $k'$ such that
$c_{k'}=c_{k'+1}\not=c_{k'+2}=c_{k'+3}$. If $c_{k+1}=c_{k+2}$,
then set $k'=k+1$. Otherwise, $c_{k+1}=c_{k+4}\neq c_{k+2}=c_{k+3}$.
Since either $c_k$ or
$c_{k+5}$ is equal to $c_{k+1}=c_{k+4}$, the index
$k'$ can be set to $k$ or $k+2$.

By symmetry, we can assume in the remainder
that $k'=1$, $c_1=c_2=0$ and $c_3=c_4=1$.
Color the edge $v_1v_2$ with $2$, the edge $v_2v_3$ with $4$ and
the edge $v_3v_4$ with $6$. The remaining edges are colored with $3$ and $5$
in the alternating way (see Figure~\ref{fig-main}).
We have obtained a proper coloring of $C$. As we can extend the coloring
of the edges of $M$ to all cycles of $F$, the resulting $(7,2)$-edge-coloring
witnesses that the circular chromatic index of $G$ does not exceed $7/2$.
\end{proof}
\begin{figure}
\begin{center}
\epsfbox{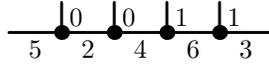}
\end{center}
\caption{Coloring odd cycles in the proof of Theorem~\ref{thm-main}.}
\label{fig-main}
\end{figure}

Petersen's theorem~\cite{bib-petersen91} asserts that
every cubic bridgeless graph has a perfect
matching; this yields the next corollary of Theorem~\ref{thm-main}.

\begin{corollary}
\label{cor-bridgeless}
The circular chromatic index of every cubic bridgeless graph
with girth $6$ or more is at most $7/2$.
\end{corollary}

Finally, we show that the assumption that the given graph is cubic
can be relaxed in Corollary~\ref{cor-bridgeless}.

\begin{corollary}
\label{cor-main}
The circular chromatic index of every subcubic graph
with odd-girth $7$ or more is at most $7/2$.
\end{corollary}

\begin{proof}
Let $G$ be a subcubic graph with odd-girth $7$ or more that has a circular
chromatic index greater than $7/2$ and that has the smallest number of
vertices among all such graphs. Consequently, the minimum degree of $G$
is at least $2$. Similarly, $G$ is connected.
The graph $G$ is also bridgeless: otherwise, each of the two
graphs obtained from $G$ by splitting along the bridge has
a $(7,2)$-edge-coloring and these edge-colorings (after rotating
the colors if necessary) combine to a $(7,2)$-edge-coloring of $G$.

If $G$ has no vertices of degree $2$, then $G$ is a cubic bridgeless
graph. Petersen's theorem~\cite{bib-petersen91} ensures that $G$ has a
a $2$-factor $F$. By our assumption, no cycle of $F$ has length $3$ or
$5$, and therefore $G$ cannot be a counter-example by
Theorem~\ref{thm-main}.
So, assume that $G$ has at least one vertex of degree $2$.
Let $H$ be a $3$-edge-connected cubic graph of odd-girth at least $7$ from which we remove an
edge.
We construct the graph $G'$ as follows.
We take two disjoint copies of $G$. For each pair $(u,v)$ of
corresponding vertices of degree $2$ (one in each copy of $G$),
we add a copy of $H$, join $u$
to a vertex of degree $2$ of $H$, and $v$ to the other vertex of
degree $2$ in $H$.
The resulting graph
$G'$ is cubic and has odd-girth at least $7$. Moreover, since $G$ is
bridgeless, $G'$ has at most two bridges. More precisely, $G'$ has two
bridges if and only if $G$ has exactly one vertex of degree $2$, and $G'$ is
bridgeless otherwise. Therefore,
$G'$ has a perfect matching by Tutte's theorem~\cite{LoPl86,Tut45}.
Consequently, since $G'$ has odd-girth at least
$7$, Theorem~\ref{thm-main} implies the existence of a
$(7,2)$-edge-coloring of $G'$. This edge-coloring restricted to
$G$ yields a $(7,2)$-edge-coloring of $G$, a contradiction.
\end{proof}

\section*{Acknowledgment}

The first and the last authors would like to thank Mohammad Ghebleh and
Luke Postle for discussions and insights on circular edge-colorings
of cubic bridgeless graphs.

\end{document}